\documentclass[11pt,reqno]{amsart}
\usepackage{amsmath,amsthm,amscd,amsfonts,amssymb,color}
\usepackage{cite}
\usepackage[mathscr]{eucal}

\usepackage[bookmarksnumbered,colorlinks,plainpages]{hyperref}
\newtheorem{theorem}{Theorem}[section]

\newtheorem*{Acknowledgement}{\textnormal{\textbf{Acknowledgement}}}

\newtheorem{proposition}[theorem]{Proposition}
\newtheorem{corollary}[theorem]{Corollary}
\newtheorem{note}[theorem]{Note}
\theoremstyle{definition}
\newtheorem{definition}[theorem]{Definition}
\newtheorem{example}[theorem]{Example}
\newtheorem{Open Prob}[theorem]{Open Problem}
\theoremstyle{remark}
\newtheorem{remark}[theorem]{Remark}
\numberwithin{equation}{section}
\def\DJ{\leavevmode\setbox0=\hbox{D}\kern0pt\rlap{\kern.04em\raise.188\ht0\hbox{-}}D}
\begin{document}

\title[Farthest Point Problem in Normed Linear Spaces]{Farthest Point Problem and Partial Statistical Continuity in Normed Linear Spaces}

\author[S.\ Som, L.K.\ Dey, S. Basu]
{Sumit Som$^{1}$, Lakshmi Kanta Dey$^{2}$, Sudeshna Basu$^{3}$}

 \address{{$^{1}$} Sumit Som,
                    Department of Mathematics,
                    National Institute of Technology
                    Durgapur, India.}
                    \email{somkakdwip@gmail.com}

\address{{$^{2}$} Lakshmi Kanta Dey,
                    Department of Mathematics,
                    National Institute of Technology
                    Durgapur, India.}
                    \email{lakshmikdey@yahoo.co.in}
\address{{$^{3}$}   Sudeshna Basu,
                    Department of Mathematics,
                    George Washington University,
                    Washington DC 20052 USA and
                    Department of Mathematics, 
                    Ram Krishna Mission Vivekananda Education and Research Institute , 
                Belur Math,  Howrah 711202
                West Bengal, India}
                    \email{sbasu@gwu.edu,sudeshnamelody@gmail.com}


\date{}

\setcounter {page}{1}



\begin{abstract}{In this paper, we  prove that if $E$ is a uniquely remotal subset of a real normed linear space $X$ such that $E$ has a Chebyshev center $c \in X$ and the farthest point map $F:X\rightarrow E$ restricted to $[c,F(c)]$ is partially statistically continuous at $c$, then $E$ is a singleton. We obtain a necessary condition on uniquely remotal subsets of uniformly rotund Banach spaces to be a singleton. Moreover, we show that there exists a remotal set $M$ having a Chebyshev center $c$ such that the farthest point map $F:\mathbb{R}\rightarrow M$ is not continuous at $c$ but is partially statistically continuous there in the multivalued sense.}
\end{abstract}

\maketitle 

\section{Introduction}

\baselineskip .82cm \small

Let $\mathbb{X}$ be a real normed linear space and $G$ be a nonempty, bounded subset of $X$. Throughout our discussion, we consider only bounded subsets of $X.$ For any $x\in X$, the farthest distance from $x$ to a set $G$ is denoted by $\delta(x,G)$ and is defined by $$\delta(x,G)=~\mbox{sup}\Big\{\Vert x-e \Vert : e\in G\Big\}.$$ If the distance is attained, then the collection of all such points of $G$ corresponding to $x\in \mathbb{X}$ is denoted by $F(x,G)$ and defined by $$F(x,G)=\Big\{e\in G : \Vert x-e \Vert = \delta(x,G)\Big\}.$$ 

For a non-empty and bounded subset $G$ of $X,$ let us define $$r(G)=\Big\{x\in X : F(x,G)\neq \phi\Big\}.$$
$G$ is said to be remotal if $r(G)=\mathbb{X}$ and uniquely remotal if $r(G)=\mathbb{X}$ and $F(x,G)$ is singleton for each $x\in X.$ A nonempty, bounded subset $G$ of $\mathbb{X}$ is said to be densely remotal if $r(G)$ is norm dense in $\mathbb{X}.$ 
 The farthest point problem (FPP) states that ``Must every uniquely remotal set in a Banach space be a singleton?" The FPP was proposed by Motzkin, Starus and Valentine \cite{mo} in the context of the Euclidean space $E^{n}$. The FPP for  Banach spaces was introduced by Klee \cite{kle} and  he proved that every compact uniquely remotal subset of a Banach space is a singleton. In \cite{ass}, Asplund solved the FPP for any finite dimensional Banach space with respect to a norm which is not necessarily symmetric. Very recently Yosef, Khalil and Mutabagani \cite{ykm} proved the FPP for $\ell_{1}$.  In their work \cite {sab}, Sababheh and Khalil proved that a closed and bounded set is remotal in $X$ if and only if $X$ is finite dimensonal. The question was answered in the negative for infinite dimensional spaces by Martin and Rao in  \cite{mr}, where they  showed  the existence of a closed bounded convex non remotal set for every infinite dimensional Banach space. Another  concept, strongly remotal set,  was introduced in \cite{km} by Khalil and Matar, and they showed that all such sets are singleton. Recall that, a Chebyshev center of a subset $E$ of a normed linear space $\mathbb{X}$ is an element $c\in \mathbb{X}$ such that $\delta(c,E)=\inf_{x \in \mathbb{X}}\delta(x,E).$ Chebyshev centers of sets have played a major role in the study of uniquely remotal sets, see \cite{ass}, \cite{na} for more details. In \cite{ni}, it was proved that if $E$ is a uniquely remotal subset of a normed space $\mathbb{X},$ admitting a Chebyshev center at $c \in X $ and if the farthest point map $F:X\rightarrow E$ restricted to $[c,F(c)]$ is continuous at $c$, then $E$ is a singleton.

Now,  we give a brief motivation of our paper. It is clear that, if $\mathbb{X}$ is a normed linear space and $M$ is a non-empty, bounded subset of $\mathbb{X},$ then the farthest point map $F:\mathbb{X}\rightarrow M$ is not always a single valued map, even if  $M$ is a remotal subset of $\mathbb{X}.$ The farthest point map $F:\mathbb{X}\rightarrow M$ is single valued only when $M$ is uniquely remotal. In next section, we first introduce the notion of partial statistical continuity of a single valued function and provide an example to show that the notion of partial statistical continuity is much weaker than continuity as well as partial continuity introduced by Sababheh et al. in \cite{sab}. We prove that if $M$ is an uniquely remotal subset of a real normed linear space $\mathbb{X}$ such that $M$ has a Chebyshev center $c$ and the farthest point map $F:X\rightarrow M$ restricted to $[c,F(c)]$ is partially statistically continuous at $c$ then $E$ is a singleton. This improves the result in \cite{ni}. We also introduce the notion of partial statistical continuity of a multivalued mapping. 
\section{\textbf{MAIN RESULTS}}

 We first recall some definitions and notations.

\begin{definition} \cite[\, Definition 2.1.]{bor} \label{borr}
Let $Y,Z$ be topological spaces. A mapping $F:Y \rightarrow Z$ is said to be multivalued if $F(x)$ is a subset of $Z$ for each $x\in Y.$ The mapping $F$ can be thought of as a single valued function from $Y$ into $2^{Z}$ where $2^{Z}$ denotes the power set of $Z.$
\end{definition}

\begin{definition} \cite[\, Definition 2.2.]{bor} \label{bor1}
Let $F:Y\rightarrow Z$ be a multivalued mapping and $A \subset Y, ~B \subset Z.$ Then 
\begin{enumerate}
\item[(i)] $F(A)=\bigcup \{F(x): x\in A\}$;

\item[(ii)] $F^{-1}(B)=\{x\in X: F(x)\cap B \neq \phi\}$.
\end{enumerate}
\end{definition}

\begin{definition}  \cite[\, Definition 2.3.]{bor} \label{bor2}
Let $F:Y\rightarrow Z$ be a multivalued mapping. Then 
\begin{enumerate}
\item[(i)] $F$ is a usc-function (i.e, upper semi-continuous function) provided $F^{-1}(B)$ is closed in $Y$ for each closed $B\subset Z$;

\item[(ii)] $F$ is a lsc-function (i.e, lower semi-continuous function) provided $F^{-1}(V)$ is open in $Y$ for each open $V\subset Z$;

\item[(iii)] $F:Y\rightarrow Z$ is a continuous function provided that $F$ is a usc-function and lsc-function.
\end{enumerate}
\end{definition}

\begin{definition} \cite{tk}
Let $F:Y\rightarrow Z$ be a multivalued mapping. Then a function $f:Y\rightarrow Z$ is said to be a selection or selector of $F:Y\rightarrow Z$ if $f(x)\in F(x)$ for each $x\in Y.$ Throughout this paper we call a selection by the name extracted single valued function. 
\end{definition}
 
We have the following proposition, we omit the easy proof.

\begin{proposition}\label{mv}
If each extracted single valued function $f:Y\rightarrow Z$ is continuous then the multivalued map $F:Y\rightarrow Z$ is continuous. 
\end{proposition}

We recall the following defnition from \cite{sab}:

\begin{definition}
Let $\mathbb{X}$ be a real normed linear space and $M \subseteq \mathbb{X}.$ A function $F:M\rightarrow \mathbb{X}$ is said to be partially continuous at $a\in M$ if there exists a non constant sequence $\{x_n\}_{n\in \mathbb{N}}\subset M$ such that $\{x_{n}\}_{n\in \mathbb{N}}$ is convergent to $a$ and $\{F(x_n)\}_{n\in \mathbb{N}}$ is convergent to $F(a).$ 
\end{definition}

The notion of usual convergence of a real sequence has been extended to statistical convergence by Fast \cite{s} in the year 1980. Further for more details about statistical convergence one can see \cite{s}. We recall the following definition from \cite{s} as follows:

\begin{definition}
A real sequence $\{x_{n}\}_{n\in \mathbb{N}}$ is said to be statistically convergent to $x\in \mathbb{R}$ if for each $\varepsilon>0,$ 
$$\displaystyle{\lim_{n\rightarrow \infty}}\frac{1}{n} \Big |\Big\{k\leq n : |x_{k}-x| \geq \varepsilon\Big\} \Big |=0.$$ 
\end{definition}

Now we like to introduce the notion of partial statistical continuity for a single valued function in a normed linear space as follows:

\begin{definition}
Let $\mathbb{X}$ be a real normed linear space and $M \subseteq \mathbb{X}.$ A function $G:M\rightarrow \mathbb{X}$ is said to be partially statistically continuous at $a\in M$ if there exists a non-constant sequence $\{x_n\}_{n\in \mathbb{N}}\subset M$ such that $\{x_{n}\}_{n\in \mathbb{N}}$ is statistically convergent to $a$ and $\{G(x_n)\}_{n\in \mathbb{N}}$ is statistically convergent to $G(a)$ i.e. for each $\varepsilon>0,$
$$\displaystyle{\lim_{n\rightarrow \infty}}\frac{1}{n}\Big|\Big\{k\leq n : \Vert x_{k}-a \Vert \geq \varepsilon\Big\}\Big|=0$$ and 
$$\displaystyle{\lim_{n\rightarrow \infty}}\frac{1}{n}\Big|\Big\{k\leq n : \Vert G(x_{k})-G(a) \Vert \geq \varepsilon\Big\}\Big|=0.$$ 
If $G:M\rightarrow \mathbb{X}$ is partially statistically continuous at each $x\in M$ then $G$ is said to be partially statistically continuous on $M.$
\end{definition}

The next example shows that, the notion of partial statistical continuity is much weaker than continuity as well as partial continuity.

\begin{example} \label{ex2}
Let $f: [-1,0]\rightarrow \mathbb{R}$ be defined by $f(x)=[x]~\mbox{for all}~ x\in [-1,0].$ Here $[x]$ denotes the greatest integer not exceeding $x.$ It is easy to check that, this function is not partially continuous (also not continuous) at the point $x=0.$ Now we show that this function is partially statistically continuous at the point $x=0.$ 
Let us define a sequence $\{x_{n}\}_{n\in \mathbb{N}}$ in $[-1,0]$ by
\begin{equation*}
~x_{n}=
\begin{cases}
0  ~~\mbox{if} ~ n\neq m^{2}~\mbox{for all}~ m\in \mathbb{N}, ~&\\
-1+\frac{1}{n}~  ~\mbox{if} ~ n=m^{2}~\mbox{for some}~ m\in \mathbb{N}.
\end{cases}
\end{equation*}
Note that, the sequence $\{x_n\}_{n\in \mathbb{N}}$ is not convergent to $0$ in the usual sense. Let $\varepsilon>0.$ 
Now $\Big\{k\in \mathbb{N}: |x_{k}-0|\geq \varepsilon\Big\}\subseteq \Big\{k\in \mathbb{N}: k=m^{2}~\mbox{for some}~m\in \mathbb{N}\Big\}.$ But the natural density of the set $A=\Big\{k\in \mathbb{N}: k=m^{2}~\mbox{for some}~m\in \mathbb{N}\Big\}$ is $0$ since,
$$\frac{1}{n}\Big|\{k\leq n : k\in A\}\Big|=\frac{[\sqrt{n}]}{n}\leq \frac{\sqrt{n}}{n}\rightarrow 0~\mbox{as}~n\rightarrow \infty.$$
This implies that
$$\displaystyle{\lim_{n\rightarrow \infty}}\frac{1}{n}\Big|\{k\leq n : \vert x_{k}-0 \vert \geq \varepsilon\}\Big|=0.$$ 
So, the sequence $\{x_{n}\}_{n\in \mathbb{N}}$ is statistically convergent to $0.$

Now the sequence $\{f(x_{n})\}_{n\in \mathbb{N}}$ is defined by
\begin{equation*}
~f(x_{n})=
\begin{cases}
0  ~\mbox{if} ~ n\neq m^{2}~\mbox{for all}~ m\in \mathbb{N} ~&\\
-1  ~\mbox{if} ~ n=m^{2}~\mbox{for some}~ m\in \mathbb{N}.
\end{cases}
\end{equation*}
Similarly, it can be shown that the sequence $\{f(x_{n})\}_{n\in \mathbb{N}}$ is statistically convergent to $f(0)=0.$ So, $f$ is partially statistically continuous at the point $x=0.$ So, the notion of partial statistical continuity is much weaker than partial continuity as well as continuity.
\end{example}

We now  introduce the concept of partial statistical continuity for a multivalued mapping.

\begin{definition}
Let $\mathbb{X}$ and $\mathbb{Y}$ be real normed linear spaces and $F: \mathbb{X}\rightarrow \mathbb{Y}$ be a multivalued mapping. The mapping $F: \mathbb{X}\rightarrow \mathbb{Y}$ is said to be partially statistically continuous on $\mathbb{X}$ if each extracted single valued function $f:\mathbb{X}\rightarrow \mathbb{Y}$ of $F$ is partially statistically continuous on $\mathbb{X}.$
\end{definition}

\begin{note}
Let $\mathbb{X}$, $\mathbb{Y}$ be two normed linear spaces. It is to be noted that, if any one of the extracted single valued functions $f:\mathbb{X}\rightarrow \mathbb{Y}$ of a multivalued mapping $F:\mathbb{X}\rightarrow \mathbb{Y},$ fail to be partially statistically continuous at some point of $\mathbb{X},$ then we cannot say anything about the continuity of the multivalued mapping $F:\mathbb{X}\rightarrow \mathbb{Y}$ at that point.
\end{note}

 We show that there exists a remotal set $M$ in the set of all real numbers $\mathbb{R}$ having a Chebyshev center $c\in \mathbb{R}$ such that the farthest point map $F:\mathbb{R}\rightarrow M$ is not continuous at $c\in \mathbb{R}$ but is partially statistically continuous there in the multivalued sense. The following result will be needed in the upcoming example.

\begin{theorem} \cite{kle} \label{kle}
Let $M \subseteq \mathbb{X}.$ If $M$ is compact and uniquely remotal then $M$ is singleton. 
\end{theorem}

\begin{example} \label{ex3}
Let $\mathbb{R}$ be the set of all real numbers and $M= [-1,1].$ It was proved in \cite{sabb}, that a closed and bounded set in a finite dimensional space is remotal, so it follows that $M$ is remotal in $\mathbb{R}$. But $M= [-1,1]$ cannot be uniquely remotal. Because if $M= [-1,1]$ is uniquely remotal, then by Theorem \ref{kle}, $M$ will be a singleton. In this case,
$$\delta(0,M)=\mbox{sup}\{\vert x \vert : x\in M\}=1.$$  
Also, $F(0,M)=\{-1,1\}.$ So, in this case, the farthest point map $F:\mathbb{R}\rightarrow M$ is multivalued. 
We show that $c=0$ is a Chebyshev center of $M.$ 

First of all suppose that $x>1.$ So $\vert x-0 \vert>1.$ In this case $\delta(x,M)\geq \vert x-0 \vert>1.$ Now let $x<-1.$ So $x=-y$ for some $y>1.$ In this case $\delta(x,M)\geq \vert x-0 \vert=\vert y \vert>1.$ Now suppose that $x\in M$ and $x>0.$ In this case $\delta(x,M)=\vert x-(-1)\vert= \vert x+1 \vert>1.$ Similarly $\delta(x,M)>1$ for $x<0.$ But $$\delta(0,M)=\mbox{inf}_{x\in \mathbb{R}}\delta(x,M)=1.$$
So $c=0$ is the Chebyshev center of $M.$ 

Now, we show that the multivalued map $F:\mathbb{R}\rightarrow M$ is not continuous, in the sense of the Definition \ref{bor2}. Here, we consider that $M$ has the subspace topology as a subset of $\mathbb{R}.$ So the set, $V=(-1,1]$ is open in $M.$ But $F^{-1}(V)=(-\infty,0],$ which is not open in $\mathbb{R}.$ This shows that $F$ is not lower-semi continuous. Hence $F$ is not continuous. Now we show that the farthest point map $F:\mathbb{R}\rightarrow M$ is partially statistically continuous at $c=0.$ The map $F$ has two extracted single valued functions. Let $F^{*}:\mathbb{R}\rightarrow M$ be the extracted single valued function such that $F^{*}(0)=1$ and $F^{**}:\mathbb{R}\rightarrow M$ be the extracted single valued function such that $F^{**}(0)=-1.$ We show that both $F^{*}$ and $F^{**}$ are partially statistically continuous at $c=0.$

Let us define a sequence $\{x_{n}\}_{n\in \mathbb{N}}$ in $[-1,1]$ by
\begin{equation*}
~x_{n}=
\begin{cases}
0  ~\mbox{if} ~~ n \neq m^{2}~\mbox{for all}~ m\in \mathbb{N},  ~&\\
-1+\frac{1}{n}  ~\mbox{if} ~n= m^{2}~\mbox{for some}~ m\in \mathbb{N}.
\end{cases}
\end{equation*}
In this case, the sequence $\{x_n\}_{n\in \mathbb{N}}$ is statistically convergent to $0$ since for each $\varepsilon>0,$
$$\displaystyle{\lim_{n\rightarrow \infty}}\frac{1}{n}\Big|\Big\{k\leq n : \vert x_{k}-0 \vert \geq \varepsilon\Big\}\Big|=0.$$
Now $F^{*}(x_n)=1$ for all $n\in \mathbb{N}.$ So $\{F^{*}(x_n)\}_{n\in \mathbb{N}}$ is statistically convergent to $F^{*}(0)=1.$ This shows that the extracted single valued function $F^{*}:\mathbb{R}\rightarrow M$ is partially statistically continuous at $c=0.$ In a similar way, we can show that the extracted single valued function $F^{**}:\mathbb{R}\rightarrow M$ is partially statistically continuous at $c=0.$ Hence the farthest point map $F:\mathbb{R}\rightarrow M$ is partially statistically continuous at $c=0.$

\end{example}

\begin{theorem}\label{ur}
Let $E \subset \mathbb{X}.$ Suppose  $E$ is uniquely remotal and  $E$ has a  Chebyshev center $c \in X .$ If the farthest point map $F:X\rightarrow E$ restricted to $[c,F(c)]$ is partially statistically continuous at $c,$ then $E$ is singleton.
\end{theorem}

\begin{proof}
Since $E$ is uniquely remotal, so for each $x\in \mathbb{X}$ there exists unique $e\in E$ such that $||x-e||=\delta(x,E)$ and the farthest point map $F: \mathbb{X}\rightarrow E $ is well defined. If $E$ has a Chebyshev center at $c=0$ then the set $E-\{c\}=\{e-c : e\in E\}$ has Chebyshev center at $c=0.$ So, without loss of generality, we assume that $E$ has Chebyshev center at $c=0.$

Suppose $E$ is not singleton. So we have $F(0)\neq 0.$ It is given that the farthest point map $F:X\rightarrow E$ restricted to $[0,F(0)]$ is partially statistically continuous at $0.$ So there exists a non constant sequence $\{x_n\}_{n\in \mathbb{N}}$ in $[0,F(0)]$ such that $\{x_n\}_{n\in \mathbb{N}}$ is statistically convergent to $0$ and $\{F(x_n)\}_{n\in \mathbb{N}}$ is statistically convergent to $F(0).$

Since $x_{n}\in [0,F(0)]$ so we have $x_{n}=\mu_{n}F(0)$ with $\mu_{n}>0~\forall~n\in \mathbb{N}$ and $\{\mu_{n}\}_{n\in \mathbb{N}}$ is statistically convergent to $0$ as $n\rightarrow \infty.$ Now by Hahn-Banach theorem for real normed linear spaces, for each $n\in \mathbb{N}$ there exists $\psi_{n} \in \mathbb{X}^{*}$ such that $\psi_{n}(F(x_n)-x_n)=||F(x_n)-x_n||$ and $||\psi_n||=1.$ Now,
\begin{eqnarray*}
\psi_{n}(x_n)&=&\psi_n(F(x_n))-\psi_{n}(F(x_n)-x_n)\\
&\leq& ||\psi_{n}||~||F(x_n)||-||F(x_n)-x_n||\\
&=& ||F(x_n)||-||F(x_n)-x_n||\\
&=& ||F(x_n)-0||-||F(x_n)-x_n||\\
&\leq& \delta(0,E)-\delta(x_n,E)\\
&\leq& 0~~(\mbox{as 0 is the Chebyshev center}).
\end{eqnarray*}

So, $$\psi_{n}(x_n)\leq 0~\mbox{for all}~n\in \mathbb{N}$$
$$\Longrightarrow \psi_{n}(\mu_{n}F(0))\leq 0~\mbox{for all}~n\in \mathbb{N}$$
$$\Longrightarrow \mu_{n}\psi_{n}(F(0))\leq 0~\mbox{for all}~n\in \mathbb{N}$$
$$\Longrightarrow \psi_{n}(F(0))\leq 0~\mbox{for all}~n\in \mathbb{N}~\mbox{as}~\mu_{n}>0.$$

Since the sequence $\{F(x_n)-x_n\}_{n\in \mathbb{N}}$ is statistically convergent to $F(0)$ so for each $\varepsilon>0,$ $$\displaystyle{\lim_{n\rightarrow \infty}}\frac{1}{n} \Big|\{k\leq n: ||F(x_k)-x_k-F(0)||\geq \varepsilon\}\Big|=0.$$

Now $$\Big|\Vert F(x_k)-x_k \Vert - \Vert F(0)\Vert \Big|\leq \Big\Vert F(x_k)-x_k-F(0)\Big\Vert.$$

Let $\varepsilon>0.$  So we have 
$$ \Big\{k\in \mathbb{N}: \Big|\Vert F(x_k)-x_k\Vert - \Vert F(0)\Vert \Big|\geq \varepsilon \Big\}\subset \Big\{k\in \mathbb{N}: \Big\Vert F(x_k)-x_k-F(0)\Big\Vert \geq \varepsilon \Big\}.$$ This implies that
$$\displaystyle{\lim_{n\rightarrow \infty}}\frac{1}{n} \Big|\{k\leq n: \Big|\Vert F(x_k)-x_k\Vert -\Vert F(0)\Vert \Big|\geq \varepsilon\}\Big|=0.$$
So the sequence $\{||F(x_k)-x_k||\}_{n\in \mathbb{N}}$ is statistically convergent to $||F(0)||.$ 

Also
\begin{eqnarray*}
\psi_{n}(F(x_n)-x_n)-\psi_{n}(F(0))&=& \psi_{n}(F(x_n)-x_n-F(0))\\
&\leq& ||\psi_{n}||~||F(x_n)-x_n-F(0)||\\
&=& ||F(x_n)-x_n-F(0)||.
\end{eqnarray*}
Since the sequence in the right is statistically convergent to $0,$ so the sequence $\{\psi_{n}(F(0))\}_{n\in \mathbb{N}}$ is statistically convergent to $||F(0)||.$ But this is possible only when $F(0)=0.$ This is a contradiction. This proves that the uniquely remotal set $E$ is singleton.
\end{proof}

\begin{remark}
Theorem \ref{ur} is a generalization and  improvement of Proposition 5 in \cite{ni}.
\end{remark}

\begin{definition} 
Recall, $X$ is a uniformly rotund space if for any   $ 0<\epsilon \leq2, $ there exists some $\delta > 0$  such that for any two vectors 
$x, y \in X, $ such that $\|x\|=\|y\|=1,$ the condition $\|x-y\|\geq \epsilon $ implies $\|\frac{x+y}{2}\| \leq 1-\delta .$ The most natural examples of such spaces are Hilbert spaces and real $L^{p}$ spaces for $1< p< \infty.$
\end{definition}
 
We need the following result due to Klee and Garkavi.

\begin{theorem} \label{klee garkavi}\cite{hh} 
Let $X$ be a uniformly rotund Banach Space and $G \subseteq X$ be bounded, then $G$  has a unique Chebyshev center. 
\end{theorem}

\begin{corollary}
Let $X$ be a uniformly rotund Banach Space and $G \subseteq X.$ If $G$ is uniquely remotal and the farthest point map $F:\mathbb{H}\rightarrow G$ is partially statistically continuous, then $G$ is a singleton.
\end{corollary}

\begin{proof}
The proof follows immediately from Theorem \ref{klee garkavi} and Theorem \ref{ur}.
\end{proof}

We recall the following well known definition from \cite{sain}. 

\begin{definition}
Let $\mathbb{X}$ be a real normed linear space and $M$ be a nonempty, bounded subset of $\mathbb{X}$. A sequence $\{x_n\}_{n\in \mathbb{N}}\subset M$ is said to be maximizing if there exists $x\in \mathbb{X}$ such that $\Vert x_n-x \Vert \rightarrow \delta(x,M)$ as $n\rightarrow \infty.$
\end{definition}

Likewise, we define,

\begin{definition}
Let $\mathbb{X}$ be a real normed linear space and $M$ be a nonempty, bounded subset of $\mathbb{X}$. A sequence $\{x_n\}_{n\in \mathbb{N}}\subset M$ is said to be statistically maximizing if there exists $x\in \mathbb{X}$ such that $\Big\{\Vert x_n-x \Vert\Big\}_{n \in \mathbb{N}}$ is statistically convergent to $\delta(x,M)=\mbox{sup}\Big\{\Vert x-y \Vert : y\in M\Big\}$ as $n\rightarrow \infty.$
\end{definition}

\begin{proposition}\label{aa}
Let $\mathbb{X}$ be a real normed linear space and $M$ be a nonempty bounded subset of $\mathbb{X}$. If $\{x_n\}_{n\in \mathbb{N}}$ is maximizing in $M$ then $\{x_n\}_{n\in \mathbb{N}}$ is statistically maximizing in $M.$ 
\end{proposition}

\begin{proof}
Let the sequence $\{x_n\}_{n\in \mathbb{N}}$ is maximizing. So there exists $x\in \mathbb{X}$ such that $\Vert x_n-x \Vert \rightarrow \delta(x,M)$ as $n\rightarrow \infty.$ Let $\varepsilon>0.$ So the set, $A=\Big\{k\in \mathbb{N}: \Big\vert \Vert x_k-x \Vert - \delta(x,M)\Big\vert \geq \varepsilon\Big\}$ is finite. Thus $$\displaystyle{\lim_{n\rightarrow \infty}}\frac{1}{n} \Big|\{k\leq n : \Big| \Vert x_k-x \Vert - \delta(x,M)\Big| \geq \varepsilon\}\Big|=0.$$ This implies that $\{x_n\}_{n\in \mathbb{N}}$ is statistically maximizing. 
\end{proof}

But the converse of Proposition \ref{aa} is not true. We present an example to support our claim.

\begin{example}
Let $\mathbb{R}$ be the set of all real numbers. Let $M= [-1,1].$ Now 
$$\delta(0,M)=\mbox{sup}\Big\{\vert x \vert : x\in M\Big\}=1.$$
Let us define a sequence $\{x_{n}\}_{n\in \mathbb{N}}$ in $[-1,1]$ by
\begin{equation*}
~x_{n}=
\begin{cases}
0  ~\mbox{if} ~~ n=m^{2}~\mbox{for some}~ m\in \mathbb{N},  ~&\\
1-\frac{1}{n} ~\mbox{if} ~n\neq m^{2}~\mbox{for all}~ m\in \mathbb{N}.
\end{cases}
\end{equation*}

It can be seen that $\{x_{n}\}_{n\in \mathbb{N}}$ is not maximizing. Because in this case, for any $x\in \mathbb{R},$ the real sequence $\Big\{|x_{n}-x|\Big\}_{n\in \mathbb{N}}$ is not convergent in $\mathbb{R}.$ Now we show that $\{x_{n}\}_{n\in \mathbb{N}}=\Big\{\vert x_{n}-0 \vert\Big\}_{n\in \mathbb{N}}$ is statistically convergent to $\delta(0,M)=1.$ Let $0<\varepsilon \leq 1.$ Now

\begin{align}
\frac{1}{n} \Big\vert\{k\leq n : \vert x_{k}-1 \vert \geq \varepsilon\}\Big\vert &= \frac{1}{n} \Big([\sqrt{n}]+d\Big) \nonumber\\
&\leq \frac{\sqrt{n}}{n}+ \frac{d}{n} \nonumber
\end{align}

$$\Longrightarrow \displaystyle{\lim_{n\rightarrow \infty}}\frac{1}{n}\Big\vert\{k\leq n : \vert x_{k}-1 \vert \geq \varepsilon\}\Big\vert = 0.$$
Here $d$ is a finite positive integer. This shows that the sequence $\{x_{n}\}_{n\in \mathbb{N}}$ is statistically maximizing in $M.$ So, we can conclude that, the notion of statistically maximizing sequence is much weaker than maximizing sequence.
\end{example}

\begin{theorem}
Let $M \subseteq \mathbb{X}.$ If $\{x_n\}_{n\in \mathbb{N}}$ is a statistically maximizing sequence in $M$ then $\{x_n\}_{n\in \mathbb{N}}$ is a statistically maximizing sequence in $\overline{M}.$
\end{theorem}

\begin{proof}
Let $\mathbb{X}$ be a real normed linear space and $M$ be a nonempty bounded subset of $\mathbb{X}.$ Let $x\in \mathbb{X}.$ It can be easily seen that $\delta(x,M)=\delta(x,\overline{M}).$ Let $\{x_n\}_{n\in \mathbb{N}}$ be a statistically maximizing sequence in $M.$ So $\{x_n\}_{n\in \mathbb{N}}$ is also a sequence in $\overline{M}.$ Since $\{x_n\}_{n\in \mathbb{N}}$ is a statistically maximizing sequence so there exists $x\in \mathbb{X}$ such that $\Vert x_n-x \Vert$ is statistically convergent to $\delta(x,M).$ Since $\delta(x,M)=\delta(x,\overline{M})$ so, we have $\Vert x_n-x \Vert$ is statistically convergent to $\delta(x,\overline{M}).$ This implies $\{x_n\}_{n\in \mathbb{N}}$ is also a statistically maximizing sequence in $\overline{M}.$
\end{proof}

Now from \cite{am}, we recall the definition of an approximate fixed point sequence of a mapping defined on a non-empty subset $C$ of a normed linear space $\mathbb{X}.$

\begin{definition}
Let  $C\subseteq X.$ Let $T:C\rightarrow X$ be a mapping. A sequence $\{x_n\}_{n\in \mathbb{N}}$ in $C$ is said to be an approximate fixed point sequence (a.f.p.s in short) for $T$ if $\Vert x_{n}-T(x_n) \Vert \rightarrow 0$ as $n\rightarrow \infty.$
\end{definition}

\begin{theorem}
Let $\mathbb{X}$ be a finite dimensional real normed linear space and $C$ be a non-empty, closed, bounded subset of $X.$ Let $T:C\rightarrow C$ be a mapping and $\{x_n\}_{n\in \mathbb{N}}$ be an a.f.p.s for $T.$ Then $\{x_n\}_{n\in \mathbb{N}}$ is statistically maximizing if and only if $\{T(x_n)\}_{n\in \mathbb{N}}$ is statistically maximizing.
\end{theorem}

\begin{proof}
Let $\mathbb{X}$ be a finite dimensional real normed linear space and $C$ be a non-empty, closed, bounded subset of $X.$ Let $T:C\rightarrow C$ be a mapping and $\{x_n\}_{n\in \mathbb{N}}$ be an a.f.p.s for $T.$ So $\Vert x_{n}-T(x_n) \Vert \rightarrow 0$ as $n\rightarrow \infty.$ Since every closed bounded subsets in a finite dimensional space $\mathbb{X}$ is remotal \cite[\, Theorem B.]{sabb}, we conclude that $C$ is remotal. 

First of all suppose that $\{x_n\}_{n\in \mathbb{N}}$ is statistically maximizing. So there exists $x\in \mathbb{X}$ such that $\{\Vert x_{n}-x \Vert \}_{n\in \mathbb{N}}$ is statistically convergent to $\delta(x,C).$ Since $C$ is remotal, there exists $p\in C$ such that $\delta(x,C)=\Vert x-p \Vert.$ Since  $\Vert x_{n}-T(x_n) \Vert \rightarrow 0$ as $n\rightarrow \infty,$ we have, $\{\Vert x_n-T(x_n) \Vert\}_{n\in \mathbb{N}}$ is statistically convergent to $0.$ We show that  $x\in X$ serves our purpose. Now, 
\begin{align*}
\Big |\Vert T(x_n)-x \Vert - \Vert x-p \Vert \Big | \leq& \Big |\Vert T(x_n)-x \Vert - \Vert x_n-x \Vert\Big |+ \Big |\Vert x_n-x \Vert - \Vert x-p \Vert \Big |\\
\leq& \Vert x_n- T(x_n) \Vert + \Big |\Vert x_n-x \Vert - \Vert x-p \Vert \Big |.
\end{align*}

Now let $\varepsilon>0.$ Since $\big \{\Vert x_n-T(x_n) \Vert\big \}_{n\in \mathbb{N}}$ is statistically convergent to $0$ so for $\frac{\varepsilon}{2}>0,$ we have

$$\displaystyle{\lim_{n\rightarrow \infty}}\frac{1}{n}\Big |\Big\{k\leq n : \Vert x_k -T(x_k) \Vert \geq \frac{\varepsilon}{2} \Big\}\Big|=0.$$

Also since $\big \{\Vert x_{n}-x \Vert \big \}_{n\in \mathbb{N}}$ is statistically convergent to $\delta(x,C)=\Vert x-p \Vert$ so

$$\displaystyle{\lim_{n\rightarrow \infty}}\frac{1}{n} \Big|\Big\{k\leq n : \Big |\Vert x_k-x \Vert - \Vert x-p \Vert \Big | \geq \frac{\varepsilon}{2}\Big\}\Big|=0.$$

Now
$$\Big\{k\leq n : \Big |\Vert T(x_k)-x \Vert - \Vert x-p \Vert \Big | \geq \varepsilon \Big\}\subseteq \Big\{k\leq n : \Vert x_k - T(x_k) \Vert \geq \frac{\varepsilon}{2} \Big\}\bigcup\Big\{k\leq n : \Big |\Vert x_k-x \Vert - \Vert x-p \Vert \Big | \geq \frac{\varepsilon}{2}\Big\}$$

\begin{align*}
\Longrightarrow \frac{1}{n}\Big |\Big\{k\leq n : \Big |\Vert T(x_k)-x \Vert - \Vert x-p \Vert \Big | \geq \varepsilon \Big\}\Big|\leq & \frac{1}{n} \Big|\Big\{k\leq n : \Vert x_k - T(x_k) \Vert \geq \frac{\varepsilon}{2} \Big\}\Big|\\
+&\frac{1}{n} \Big|\Big\{k\leq n : \Big |\Vert x_k-x \Vert - \Vert x-p \Vert \Big | \geq \frac{\varepsilon}{2}\Big\}\Big|
\end{align*}

$$\Longrightarrow \displaystyle{\lim_{n\rightarrow \infty}}\frac{1}{n}\Big |\Big\{k\leq n : \Big |\Vert T(x_k)-x \Vert - \Vert x-p \Vert \Big | \geq \varepsilon \Big\}\Big|=0.$$
Which implies that the sequence $\{\Vert T(x_n)-x \Vert\}_{n\in \mathbb{N}}$ is statistically convergent to $\Vert x-p \Vert=\delta(x,C).$ So $\{T(x_n)\}_{n\in \mathbb{N}}$ is statistically maximizing. 

Proceeding similarly, one can show that  if $\{T(x_n)\}_{n\in \mathbb{N}}$ is statistically maximizing, then the sequence $\{x_n\}_{n\in \mathbb{N}}$ is statistically maximizing. 
\end{proof}

\begin{Acknowledgement}
The Research is funded by the Council of Scientific and Industrial Research (CSIR), Government of India under the Grant Number: $25(0285)/18/\mbox{EMR-II}$.
The third author is very grateful to the first and second authors for inviting her to National Institute of Technology, Durgapur and the kind hospitality extended towards  her.
\end{Acknowledgement}

\end{document}